\documentclass[reqno, 11pt]{amsart}
\usepackage{dsfont}
\usepackage{amssymb}
\usepackage{mathrsfs}
\usepackage{enumitem}
\usepackage{comment}
\usepackage{xurl}
\usepackage[backend=biber, style=alphabetic, url=false]{biblatex}
\AtBeginBibliography{\raggedright}

\usepackage[pdfstartview=FitH, pdfborder={0 0 0}, colorlinks=true, citecolor=blue, linkcolor=blue, urlcolor=blue]{hyperref}
\usepackage{aliascnt}

\theoremstyle{plain}
\newtheorem{theorem}{Theorem}[section]

\newaliascnt{conjecture}{theorem}

\aliascntresetthe{conjecture}

\newaliascnt{corollary}{theorem}
\newtheorem{corollary}[corollary]{Corollary}
\aliascntresetthe{corollary}

\newaliascnt{lemma}{theorem}
\newtheorem{lemma}[lemma]{Lemma}
\aliascntresetthe{lemma}

\newaliascnt{fact}{theorem}

\aliascntresetthe{fact}

\newaliascnt{claim}{theorem}

\aliascntresetthe{claim}

\newaliascnt{proposition}{theorem}

\aliascntresetthe{proposition}

\theoremstyle{definition}

\newaliascnt{definition}{theorem}

\aliascntresetthe{definition}

\newaliascnt{example}{theorem}

\aliascntresetthe{example}

\theoremstyle{remark}

\newaliascnt{remark}{theorem}
\newtheorem{remark}[remark]{Remark}
\aliascntresetthe{remark}

\numberwithin{equation}{section}
\DeclareMathOperator{\Ric}{Ric}

\DeclareMathOperator{\Vol}{Vol}
\DeclareMathOperator{\tr}{tr}
\DeclareMathOperator{\Pf}{Pf}
\DeclareMathOperator{\Id}{Id}
\DeclareMathOperator{\sgn}{sgn}
\newcommand{\CP}{\mathbb{CP}}

\newcommand{\sca}{\mathrm{Scal}}
\begin{document}

\title[Total scalar curvature]{Total scalar curvature under a curvature operator lower bound}
\author[J. Ge]{Jian Ge}
\address[Ge]{School of Mathematical Sciences, Laboratory of Mathematics and Complex Systems, Beijing Normal University, Beijing 100875, P. R. China.}
\email{jge@bnu.edu.cn}
\thanks{NSFC 12371049 and the Fundamental Research Funds for the Central Universities.}

\author[C. Li]{Chuanhuan Li}
\address[C. Li]{Shanghai Institute for Mathematics and Interdisciplinary Sciences (SIMIS), Shanghai 200433, China \newline
${\quad}$ Research Institute of Intelligent Complex Systems, Fudan University, Shanghai 200433, China}
\email{chli@simis.cn}
\thanks{No.2026M793350}

\author[R. Li]{Ronggang Li}
\address[R. Li]{School of Mathematics and Statistics, Nanjing University of Information Science and Technology, Nanjing, China}
\email{003705@nuist.edu.cn}
\thanks{Open Project Program of the Key Laboratory of Mathematics and Complex System at Beijing Normal University (Grant No. 202501)}

\begin{abstract}
	Let $(M^n, g)$ be a complete, simply connected Riemannian manifold without boundary, of dimension $n\ge3$, with curvature operator at least that of the unit sphere. We prove that
	\begin{equation*}
		\int_M\sca(x)\,d\Vol_x\le n(n-1)\omega_n,
	\end{equation*}
	where $\omega_n$ is the volume of the unit $n$-sphere. Equality holds if and only if $(M,g)$ is isometric to the unit round sphere. In fact, we obtain a stronger bound containing $\Vol(M, g)$. In even dimensions, the proof follows from the Chern--Gauss--Bonnet formula. In odd dimensions, we apply the corresponding boundary formula to Deruelle's Ricci expander filling.
\end{abstract}
\maketitle

\section{Introduction}\label{sec:introduction}

Comparison geometry shows that a positive lower curvature bound controls the diameter and volume of a Riemannian manifold. For instance, if a complete $n$-dimensional manifold satisfies $\sec_g\ge1$, then the Bonnet--Myers theorem and Bishop--Gromov comparison give
\begin{equation*}
	\operatorname{diam}(M,g)\le\pi,
	\qquad \Vol(M,g)\le\omega_n,
\end{equation*}
where $\omega_n$ denotes the volume of the unit round sphere $\mathbb S^n(1)$. On the other hand, larger positive sectional curvature tends to make nearby geodesics focus more rapidly. This suggests that regions of very large curvature should occupy correspondingly small volume. Understanding the balance between scalar curvature and volume is therefore a natural problem in comparison geometry.

An important result in this direction is Petrunin's estimate \cite{Pet2009}. If $(M^n,g)$ is complete and $\sec_g\ge-1$, then, for every $p\in M$,
\begin{equation*}
	\int_{B_1(p)}\sca(x)\,d\Vol_x\le C(n),
\end{equation*}
where $C(n)$ depends only on the dimension. No positive lower bound for the volume is required. For closed manifolds, this gives an upper bound for the total scalar curvature in terms of the dimension, diameter, and a lower bound for sectional curvature. In particular, under $\sec_g\ge1$, the diameter bound above gives an upper bound depending only on $n$. Li \cite{Li2025} gave an alternative proof using a finite stratified covering directly on the given manifold, see also the recent extension to Alexandrov spaces \cite{Li2026}.

These estimates lead to the question of the optimal upper bound. For $\sec_g\ge1$, Li conjectured that the unit round sphere gives the sharp bound; see \cite[Conjecture~1.14]{Li2026}. In this paper, we prove this bound under a lower bound for the curvature operator and obtain a stronger inequality that also contains the volume of $M$.

We write $\sca(x)$ for scalar curvature and $d\Vol_x$ for the Riemannian volume element. The curvature operator $R$ acts on $\Lambda^2T^*M$ and is normalized so that the unit round sphere has curvature operator $\Id$. We add a metric subscript when comparing different metrics.

\begin{theorem}\label{thm:main}
	Let $(M^n,g)$ be a complete, simply connected Riemannian manifold without boundary, of dimension $n\ge3$. If $R\ge\Id$, then $M$ is compact and
	\begin{equation}\label{eq:main-bound}
		\begin{aligned}
		\int_M\sca(x)\,d\Vol_x
		&\le2(n-1)\omega_n+(n-1)(n-2)\Vol(M,g)\\
		&\le n(n-1)\omega_n.
		\end{aligned}
	\end{equation}
	Equality in the last bound holds if and only if $(M,g)$ is isometric to the unit round sphere $\mathbb S^n(1)$.
\end{theorem}

In dimension three, every two-form is decomposable, so $R\ge\Id$ is equivalent to $\sec_g\ge1$. Thus \autoref{thm:main} gives the following sharp estimate.

\begin{corollary}\label{cor:dimension-three}
	Let $(M^3,g)$ be a complete, simply connected Riemannian manifold without boundary. If $\sec_g\ge1$, then
	\begin{equation*}
		\int_M\sca(x)\,d\Vol_x\le8\pi^2+2\Vol(M,g)\le12\pi^2.
	\end{equation*}
	Equality in the last bound holds if and only if $(M,g)$ is isometric to $\mathbb S^3(1)$.
\end{corollary}

The proof of \autoref{thm:main} uses the Chern--Gauss--Bonnet formula \cite{Che1945}. We expand the relevant curvature polynomial about the constant sectional curvature one tensor. The constant and linear terms give the volume and scalar curvature integral, while all higher-order terms are nonnegative under the curvature-operator hypothesis; compare \cite{BK1978}.

In even dimensions, the sphere theorem of B\"ohm--Wilking \cite{BW2008} implies that $M$ is diffeomorphic to $\mathbb S^n$, so $\chi(M)=2$. Integrating the expansion of the Euler density and discarding the nonnegative higher-order terms yields the first inequality in \eqref{eq:main-bound}.

In odd dimensions, the geometric idea is to realize $(M,g)$ as the cross section at infinity of a smooth manifold of one higher dimension. For $R>\Id$, Deruelle's existence theorem \cite{Der2016} provides a complete expanding Ricci soliton $(N^{n+1},G)$ with positive curvature operator, diffeomorphic to $\mathbb R^{n+1}$ and smoothly asymptotic to the metric cone $dr^2+r^2g$. Thus $(M,g)$ is isometric to the unit cross section of the tangent cone at infinity of $(N,G)$: after suitable identifications, the induced metrics on large boundaries, divided by the square of their radial parameter, converge smoothly to $g$.

The role of this filling is to provide an upper bound for the limiting boundary integral. On each smooth domain of a ball exhaustion, the classical Chern--Gauss--Bonnet formula writes the Euler characteristic $1$ as the sum of the interior Euler integral and the boundary integral. The interior term is nonnegative and is discarded. Passing to infinity therefore gives
\begin{equation*}
	\int_M b_g\,d\Vol_x\le1,
\end{equation*}
where $b_g\,d\Vol_x$ is the boundary form computed on the unit slice of the metric cone. For the scalar curvature bound, the filling is used only to establish this boundary inequality. The remaining curvature calculation takes place entirely on $(M,g)$: expanding $b_g$ gives the volume and total scalar curvature terms, together with nonnegative higher-order terms. The case $R\ge\Id$ follows by applying this argument to $g_a=ag$, $0<a<1$, and letting $a\uparrow1$.

The Euler-density expansion and the positivity of its higher-order terms are established in \autoref{sec:preliminaries}; the boundary expansion is computed in \autoref{sec:odd}. Finally, Bishop--Gromov volume comparison and its equality case give the second inequality in \eqref{eq:main-bound} and the rigidity statement.

The proof of \autoref{thm:main} is given in \autoref{sec:proof}.

\begin{remark}[The K\"ahler case]\label{rem:kahler-comparison}
	Let $g$ be a K\"ahler metric on $\CP^m$ with its standard complex structure, $m\ge2$. If $\sca(x)\ge4m(m+1)$, then
	\begin{equation*}
		\int_{\CP^m}\sca(x)\,d\Vol_x\le\frac{4m(m+1)\pi^m}{m!},
	\end{equation*}
	with equality exactly for the Fubini--Study metric $g_{\mathrm{FS}}$ of holomorphic sectional curvature $4$, up to holomorphic isometry.

	Indeed, on a closed K\"ahler manifold $(M^{2m},J,g)$ with K\"ahler form $\omega$,
	\begin{equation*}
		\int_M\sca(x)\,d\Vol_x
		=\frac{4\pi}{(m-1)!}\bigl\langle c_1(M,J)\cup [\omega]^{m-1},[M]\bigr\rangle.
	\end{equation*}
	Thus the total scalar curvature is fixed within each K\"ahler class. On $\CP^m$, let $\omega_{\mathrm{FS}}$ be the K\"ahler form of $g_{\mathrm{FS}}$. Writing $[\omega]=a[\omega_{\mathrm{FS}}]$, $a>0$, gives
	\begin{equation*}
		\Vol(\CP^m,g)=\frac{\pi^m}{m!}a^m,
		\qquad
		\int_{\CP^m}\sca(x)\,d\Vol_x=\frac{4m(m+1)\pi^m}{m!}a^{m-1}.
	\end{equation*}
	The scalar curvature lower bound therefore implies $a\le1$, proving the estimate. Equality forces $a=1$ and $\sca\equiv4m(m+1)$. The $\partial\bar\partial$-lemma then makes the Ricci potential harmonic, so $g$ is K\"ahler--Einstein. Bando--Mabuchi uniqueness \cite{BM1987} gives the stated rigidity.

	The same bound and rigidity hold for any closed K\"ahler manifold of complex dimension $m\ge2$ with $\sec_g\ge1$, by the classification theorem for positive orthogonal bisectional curvature \cite{FLW2017} and unitary averaging of the curvature tensor.
\end{remark}

\begin{remark}[Ricci curvature bounds]\label{rem:recent-ricci-counterexamples}
	The bound in \autoref{cor:dimension-three} does not extend to metrics satisfying only $\Ric_g\ge2g$. Indeed, Xu's closed-manifold construction \cite{Xu2026} can be arranged to yield, after normalization, metrics on $\mathbb S^3$ with this Ricci lower bound and arbitrarily large total scalar curvature.
\end{remark}

\section{Preliminaries}\label{sec:preliminaries}

\subsection{The Chern--Gauss--Bonnet formula and the Pfaffian expansion}\label{sec:pfaffian-expansion}

For a closed oriented Riemannian manifold $(M^{2m},g)$, the Chern--Gauss--Bonnet theorem \cite{Che1945} states that
\begin{equation}\label{eq:closed-gauss-bonnet}
	\chi(M)=\int_M e_g\,d\Vol_x,
	\qquad e_g\,d\Vol_x=\Pf\left(\frac{\Omega^R}{2\pi}\right),
\end{equation}
where $\Omega^R$ is the matrix of curvature two-forms of the Levi-Civita connection. We first expand this Pfaffian about the curvature of the unit sphere. The same algebraic coefficients will also be used on odd-dimensional boundaries, so we give their definitions in arbitrary dimension $n$.

Fix a local oriented orthonormal frame $e_1,\ldots,e_n$ with dual coframe $\theta^1,\ldots,\theta^n$, and write $\theta^{ij}=\theta^i\wedge\theta^j$. For an algebraic curvature operator $A$ on $\Lambda^2T_x^*M$, set
\begin{equation*}
	A_{ijkl}=\langle A(\theta^{ij}),\theta^{kl}\rangle,
	\qquad
	\Omega^A_{ij}=\frac12\sum_{k,l=1}^n A_{ijkl}\theta^k\wedge\theta^l
	=\sum_{k<l}A_{ijkl}\theta^{kl}.
\end{equation*}
Thus $\Omega^A_{ij}=A(\theta^{ij})$. The factor $1/2$ appears only when the sum runs over all ordered pairs $(k,l)$. Our convention gives $\sec_g(\operatorname{span}\{e_i,e_j\})=R_{ijij}$ and $\sca(x)=2\sum_{i<j}R_{ijij}$.

The identity operator $\Id$ on $\Lambda^2T_x^*M$ corresponds to the constant sectional curvature one tensor, so $\Omega^{\Id}_{ij}=\theta^{ij}$. Setting $E:=R-\Id$, we have
\begin{equation}\label{eq:curvature-splitting}
	R=\Id+E,
	\qquad \Omega^R_{ij}=\theta^{ij}+\Omega^E_{ij},
	\qquad 2\tr E=\sca(x)-n(n-1).
\end{equation}
Thus the curvature-operator hypothesis $R\ge\Id$ is equivalent to $E\ge0$.

Write $\sigma_i=\sigma(i)$ and $\varepsilon_\sigma=\sgn(\sigma)$. For $0\le j\le\lfloor n/2\rfloor$, define the homogeneous curvature polynomial $H_j(A)$ by
\begin{equation}\label{eq:mixed-contraction}
	\begin{aligned}
		&n!H_j(A)\,d\Vol_x\\
		&\quad=\sum_{\sigma\in S_n}\varepsilon_\sigma
		\Omega^A_{\sigma_1\sigma_2}\wedge\cdots\wedge
		\Omega^A_{\sigma_{2j-1}\sigma_{2j}}\wedge
		\theta^{\sigma_{2j+1}}\wedge\cdots\wedge\theta^{\sigma_n}.
	\end{aligned}
\end{equation}
Empty products have their usual meaning. These alternating contractions are independent of the chosen orthonormal frame and are the normalized Gauss--Bonnet--Weyl curvature polynomials, also called Lipschitz--Killing curvature polynomials. Directly from the definition,
\begin{equation}\label{eq:lowest-coefficients}
	H_0(A)=1,
	\qquad H_1(A)=\frac{2\tr A}{n(n-1)},
	\qquad H_j(c\Id)=c^j.
\end{equation}
Indeed, for $j=1$ only the coefficient of $\theta^{\sigma_1\sigma_2}$ in $\Omega^A_{\sigma_1\sigma_2}$ contributes, and each diagonal entry $A_{ijij}$ with $i<j$ occurs $2(n-2)!$ times. For $A=c\Id$, each signed summand in \eqref{eq:mixed-contraction} equals $c^j\,d\Vol_x$.

Now let $n=2m$. With this normalization, the Pfaffian in \eqref{eq:closed-gauss-bonnet} is
\begin{equation}\label{eq:euler-form}
	e_g\,d\Vol_x
	=\frac{1}{(2\pi)^m2^m m!}
	\sum_{\sigma\in S_n}\varepsilon_\sigma
	\Omega^R_{\sigma_1\sigma_2}\wedge\cdots\wedge
	\Omega^R_{\sigma_{n-1}\sigma_n}.
\end{equation}
Since $n!/((2\pi)^m2^m m!)=2/\omega_n$, substituting \eqref{eq:curvature-splitting} gives
\begin{equation}\label{eq:even-expansion}
	\begin{aligned}
		e_g&=\frac2{\omega_n}\sum_{j=0}^m\binom mj H_j(E)\\
		&=\frac2{\omega_n}+\frac{2\tr E}{(n-1)\omega_n}
		+\frac2{\omega_n}\sum_{j=2}^m\binom mj H_j(E)\\
		&=\frac{\sca(x)-(n-1)(n-2)}{(n-1)\omega_n}
		+\frac2{\omega_n}\sum_{j=2}^m\binom mj H_j(E).
	\end{aligned}
\end{equation}
The binomial coefficient counts the choices of the $j$ factors involving $E$: interchanging two index pairs is an even permutation, and two-forms commute under exterior multiplication. The constant and linear terms in \eqref{eq:even-expansion} give precisely the volume and total scalar curvature terms after integration. To obtain an upper bound for the latter, it remains to control the signs of the higher-order terms. Their signed expressions in \eqref{eq:mixed-contraction} do not make these signs apparent. The next lemma is where the hypothesis $E\ge0$ enters.

\subsection{Positivity of the curvature polynomials}\label{sec:mixed-coefficients}

The positivity below follows from the classical eigenform expansion of the Euler integrand; see Bourguignon--Karcher \cite[\S~7.3, p.~89]{BK1978}. We include the algebraic proof for completeness and to record our normalization. We use the Hodge operator $*$ with the convention $\alpha\wedge(*\beta)=\langle\alpha,\beta\rangle\,d\Vol_x$ for forms of the same degree.

\begin{lemma}\label{lem:mixed-positivity}
	Let $n\ge2$ and let $A\ge0$ be an algebraic curvature operator on $\Lambda^2T_x^*M$, with an orthonormal eigenbasis $\{\eta_a\}_{a=1}^{\binom n2}$ and eigenvalues $\lambda_a\ge0$. For $0\le j\le\lfloor n/2\rfloor$,
	\begin{equation}\label{eq:coefficient-square}
		H_j(A)=\frac{2^j(n-2j)!}{n!}
		\sum_{a_1,\ldots,a_j}\lambda_{a_1}\cdots\lambda_{a_j}
		|\eta_{a_1}\wedge\cdots\wedge\eta_{a_j}|^2\ge0.
	\end{equation}
	The sum is over ordered tuples, with repetitions allowed and the usual empty-product convention for $j=0$. In particular, in even dimension the Pfaffian density of a nonnegative curvature operator is nonnegative, and it is positive if the operator is positive definite.
\end{lemma}

\begin{proof}
	The case $j=0$ is immediate. For $j\ge1$, write
	\begin{equation*}
		\Omega^A_{ij}=\sum_a\lambda_a\eta_a(e_i,e_j)\eta_a.
	\end{equation*}
	For a fixed tuple set $w=\eta_{a_1}\wedge\cdots\wedge\eta_{a_j}$. The corresponding alternation in \eqref{eq:mixed-contraction} gives
	\begin{equation*}
		\sum_{\sigma\in S_n}\varepsilon_\sigma
		\prod_{r=1}^j\eta_{a_r}(e_{\sigma_{2r-1}},e_{\sigma_{2r}})
		\theta^{\sigma_{2j+1}}\wedge\cdots\wedge\theta^{\sigma_n}
		=2^j(n-2j)!\,{*}w.
	\end{equation*}
	Since $w\wedge(*w)=|w|^2\,d\Vol_x$, substitution and summation prove \eqref{eq:coefficient-square}. If $A>0$, all eigenvalues are positive and some such wedge product is nonzero, since the eigenforms form a basis of $\Lambda^2T_x^*M$. Thus $H_j(A)>0$, which gives the asserted strict positivity of the Pfaffian in the top degree.
\end{proof}

In even dimension $n=2m$, the higher-order terms in \eqref{eq:even-expansion} are nonnegative when $E\ge0$, and
\begin{equation}\label{eq:even-density-bound}
	e_g\ge\frac{\sca(x)-(n-1)(n-2)}{(n-1)\omega_n}.
\end{equation}
The same lemma applies to the mixed curvature contractions in the odd-dimensional boundary formula below.

\section{Proof of \texorpdfstring{\autoref{thm:main}}{Theorem 1.1}}\label{sec:proof}

Throughout this section, $(M^n,g)$ satisfies the hypotheses of \autoref{thm:main}. Bonnet--Myers implies that $M$ is compact. We first prove
\begin{equation}\label{eq:simply-connected-bound}
	\int_M\sca(x)\,d\Vol_x\le2(n-1)\omega_n+(n-1)(n-2)\Vol(M,g).
\end{equation}

\subsection{Even dimensions}\label{sec:even}

Let $n=2m\ge4$. Since $R\ge\Id>0$, the sphere theorem of B\"ohm--Wilking \cite{BW2008} implies that $M$ is diffeomorphic to $\mathbb S^n$, so $\chi(M)=2$. Integrating \eqref{eq:even-expansion} and using \eqref{eq:closed-gauss-bonnet} gives
\begin{equation}\label{eq:even-integral-identity}
	\begin{aligned}
		&\int_M\left(\sca(x)+2(n-1)\sum_{j=2}^m\binom mj H_j(E)\right)d\Vol_x\\
		&\qquad=2(n-1)\omega_n+(n-1)(n-2)\Vol(M,g).
	\end{aligned}
\end{equation}
Each term in the sum is nonnegative by \autoref{lem:mixed-positivity}. Discarding this sum proves \eqref{eq:simply-connected-bound}.

\subsection{Odd dimensions}\label{sec:odd}

Let $n=2m+1\ge3$, and first assume that $R>\Id$.

\subsubsection*{The filling}
By Deruelle's existence theorem \cite[Theorem~1.3]{Der2016}, there is a complete expanding gradient Ricci soliton $(N^{2m+2},G,f)$ with positive curvature operator, satisfying
\begin{equation*}
	\mathrm{Hess}_G f=\Ric_G+\frac12G,
\end{equation*}
and smoothly asymptotic to the cone $g_C=dr^2+r^2g$. In suitable coordinates near infinity,
\begin{equation}\label{eq:expander-asymptotics}
	f=\frac{r^2}{4},\qquad
	|(\nabla^{g_C})^{\ell}(G-g_C)|_{g_C}=O(r^{-2-\ell}),
	\qquad \ell\ge0.
\end{equation}
Here the normalization of $f$ and the smooth asymptotic estimates are those of \cite[Definition~1.2 and the subsequent discussion]{Der2016}. We use the strict hypothesis in the published Theorem~1.3; the non-strict case will follow by scaling.

The inequality $\mathrm{Hess}_G f\ge G/2$ gives, along a unit-speed minimizing geodesic from a fixed point $o$,
\begin{equation*}
	f(x)\ge f(o)-|\nabla^G f(o)|_Gd_G(o,x)+\frac14d_G(o,x)^2.
\end{equation*}
Thus $f$ is proper and attains its minimum. Strict convexity along geodesics implies that this minimum is its unique critical point; its Hessian is positive definite there. By the Morse lemma and the absence of other critical points, every sufficiently large sublevel
\begin{equation*}
	D_\rho=\{f\le\rho^2/4\}
\end{equation*}
is diffeomorphic to a closed ball. In particular, $\chi(D_\rho)=1$.

\subsubsection*{The boundary formula}
Let $(D^{2m+2},G)$ be a compact oriented Riemannian manifold with smooth boundary endowed with its induced orientation. Choose a local oriented orthonormal frame $e_1,\ldots,e_n$ on $\partial D$, with dual coframe $\theta^1,\ldots,\theta^n$, and let $\nu$ be the outward unit normal. Put
\begin{equation*}
	L_{ij}=G(\nabla^G_{e_i}\nu,e_j),\qquad \alpha_i=\sum_jL_{ij}\theta^j.
\end{equation*}
Thus a Euclidean unit sphere has $\alpha_i=\theta^i$. Write $\Omega^\partial_{ij}$ for the intrinsic curvature forms of the boundary and define the interpolating two-forms
\begin{equation}\label{eq:boundary-interpolation}
	\Omega_{ij}(t)=\Omega^\partial_{ij}-t^2\alpha_i\wedge\alpha_j.
\end{equation}
The classical Chern--Gauss--Bonnet boundary formula \cite{Che1945} takes the form
\begin{equation}\label{eq:boundary-gauss-bonnet}
	\chi(D)=\int_D e_G\,d\Vol_{G,x}+\int_{\partial D}T_G,
\end{equation}
where
\begin{equation}\label{eq:transgression-form}
	T_G=\frac{1}{(2\pi)^{m+1}2^m m!}
	\int_0^1\sum_{\sigma\in S_n}\varepsilon_\sigma
	\alpha_{\sigma_1}\wedge\Omega_{\sigma_2\sigma_3}(t)\wedge\cdots\wedge
	\Omega_{\sigma_{n-1}\sigma_n}(t)\,dt.
\end{equation}
Here $e_G$ is the interior Pfaffian density, with the normalization of \eqref{eq:euler-form} in dimension $2m+2$.

With the outward-normal convention above, the boundary form on the Euclidean unit sphere is $T_G=\omega_n^{-1}d\Vol_{\partial D}$, so that $\int_{\partial D}T_G=1$ for the Euclidean unit ball $D$.

\subsubsection*{The boundary density on the cone}
On the slice $r=1$ of $g_C=dr^2+r^2g$, with outward normal $\nu=\partial_r$, one has $\alpha_i=\theta^i$ and
\begin{equation*}
	\Omega_{ij}(t)=\Omega^E_{ij}+(1-t^2)\theta^{ij}.
\end{equation*}

Moving the single coframe factor in \eqref{eq:transgression-form} past two-forms introduces no sign. The contraction with $j$ factors of $\Omega^E$ is therefore $n!H_j(E)\,d\Vol_x$, as in \eqref{eq:mixed-contraction}. Expanding \eqref{eq:transgression-form} and integrating in $t$ gives $T_{g_C}=b_g\,d\Vol_x$ on this slice, where
\begin{equation}\label{eq:odd-expansion}
	\begin{aligned}
		b_g&:=\frac{n!}{(2\pi)^{m+1}2^m m!}
		\sum_{j=0}^m\binom mj H_j(E)\int_0^1(1-t^2)^{m-j}\,dt\\
		&=\frac1{\omega_n}\sum_{j=0}^m\binom{n/2}{j}H_j(E)\\
		&=\frac{\sca(x)-(n-1)(n-2)}{2(n-1)\omega_n}
		+\frac1{\omega_n}\sum_{j=2}^m\binom{n/2}{j}H_j(E).
	\end{aligned}
\end{equation}

The second line follows from $n=2m+1$ and the elementary identity
\begin{equation*}
	\int_0^1(1-t^2)^q\,dt=\frac{2^q q!}{(2q+1)!!},\qquad q=0,1,\ldots.
\end{equation*}
Here $\binom{x}{j}=x(x-1)\cdots(x-j+1)/j!$, with $\binom{x}{0}=1$. All coefficients in this range are positive. The constant and linear terms are $1/\omega_n$ and $\tr E/((n-1)\omega_n)$; the higher-order terms are nonnegative by \autoref{lem:mixed-positivity}. Consequently,
\begin{equation}\label{eq:odd-density-bound}
	b_g\ge\frac{\sca(x)-(n-1)(n-2)}{2(n-1)\omega_n}.
\end{equation}
In dimension three there are no higher-order terms, and
\begin{equation*}
	b_g=\frac{\sca(x)-2}{8\pi^2}.
\end{equation*}

\subsubsection*{Passing to infinity}
The Pfaffian of a positive semidefinite curvature operator is nonnegative by \autoref{lem:mixed-positivity}, applied in the interior dimension with the top value of $j$. In particular, $e_G\ge0$. Applying \eqref{eq:boundary-gauss-bonnet} to $D_\rho$ gives
\begin{equation}\label{eq:exhaustion-gauss-bonnet}
	\int_{\partial D_\rho}T_G
	=1-\int_{D_\rho}e_G\,d\Vol_{G,x}\le1.
\end{equation}
Let $g_\rho$ and $L_\rho$ be the induced metric and covariant second fundamental form of $\partial D_\rho$. After identifying this boundary with $M$, \eqref{eq:expander-asymptotics} implies
\begin{equation}\label{eq:boundary-convergence}
	\rho^{-2}g_\rho\longrightarrow g,\qquad
	\rho^{-1}L_\rho\longrightarrow g
\end{equation}
smoothly. Indeed, pulling the metric back to a fixed annulus by $r=\rho s$ and multiplying it by $\rho^{-2}$ gives smooth convergence to $ds^2+s^2g$. The level set $f=\rho^2/4$ is $s=1$ in these coordinates, so its induced metric and second fundamental form converge to those of the unit cone slice.

Constant rescaling leaves the integral in \eqref{eq:transgression-form} unchanged. The normalized boundary data in \eqref{eq:boundary-convergence} therefore give
\begin{equation*}
	\lim_{\rho\to\infty}\int_{\partial D_\rho}T_G=\int_M b_g\,d\Vol_x.
\end{equation*}
Taking the limit in \eqref{eq:exhaustion-gauss-bonnet} therefore gives
\begin{equation}\label{eq:boundary-integral-bound}
	\int_M b_g\,d\Vol_x\le1.
\end{equation}
Together with \eqref{eq:odd-density-bound}, this proves \eqref{eq:simply-connected-bound}. In fact, the positive curvature operator of $G$ gives $e_G>0$. Fixing a nonempty sublevel $D_{\rho_0}$ in \eqref{eq:exhaustion-gauss-bonnet} gives a positive lower bound for the interior integral for all $\rho\ge\rho_0$. Hence the boundary bound, and thus the first scalar curvature bound, is strict when $R>\Id$.

\subsection{The non-strict bound and rigidity}\label{sec:rigidity}

In odd dimensions, when only $R\ge\Id$ is assumed, apply the strict case to $g_a=ag$, $0<a<1$. Under the natural identification of curvature operators,
\begin{equation*}
	R_{g_a}=a^{-1}R\ge a^{-1}\Id>\Id.
\end{equation*}
The scaling laws $\sca_{g_a}=a^{-1}\sca$ and $d\Vol_{g_a,x}=a^{n/2}d\Vol_x$ give
\begin{equation*}
	a^{(n-2)/2}\int_M\sca(x)\,d\Vol_x
	\le2(n-1)\omega_n+(n-1)(n-2)a^{n/2}\Vol(M,g).
\end{equation*}
Letting $a\uparrow1$ proves \eqref{eq:simply-connected-bound}. The even-dimensional argument already applies directly to $E\ge0$.

Since $R\ge\Id$ gives $\Ric\ge(n-1)g$, Bishop--Gromov comparison yields $\Vol(M,g)\le\omega_n$. This proves the second inequality in \eqref{eq:main-bound}. Equality in the last bound forces $\Vol(M,g)=\omega_n$, so the equality case of volume comparison implies that $(M,g)$ is isometric to $\mathbb S^n(1)$. Conversely, the unit round sphere attains the bound. This proves \autoref{thm:main}.

For the strict curvature-operator hypothesis, the sphere constant remains the optimal supremum. Indeed, the round metric $g_k=k^{-1}g_{\mathbb S^n(1)}$ has curvature operator $k\Id$ and
\begin{equation*}
	\int_{\mathbb S^n}\sca_{g_k}(x)\,d\Vol_{g_k,x}
	=n(n-1)\omega_n k^{1-n/2}\longrightarrow n(n-1)\omega_n
	\quad\text{as }k\downarrow1.
\end{equation*}

The corresponding statements for non-simply-connected manifolds follow immediately by passing to the finite universal Riemannian covering. Both the scalar curvature integral and the volume are multiplied by the covering degree, so the sphere constant is divided by that degree. Equality then characterizes unit-curvature spherical space forms.

\subsection*{Acknowledgements}
We thank Nan Li for suggesting the problem studied in this paper. The first author thanks Guoyi Xu for helpful discussions concerning Xu’s counterexample cited above. During these discussions, he learned that Xu had also obtained related curvature integral estimates.

\noindent\textbf{AI assistance.} The authors acknowledge the use of AI tools as a conversational tool during the development and refinement of several ideas in this work, the paper was written by the three authors.

\printbibliography
\end{document}